\documentclass[a4paper,intlimits,12pt,reqno]{amsart}
\usepackage[T1]{fontenc}
\usepackage[utf8]{inputenc}

\usepackage{latexsym,amsthm,amsmath,amssymb,mathrsfs,mathtools}
\usepackage{tikz}
\usetikzlibrary{arrows,calc,patterns}
\usepackage{wrapfig}

\usepackage[mathscr]{eucal}

plus 6pt minus 12pt
\newcommand{\cM}{\mathcal M}

\newcommand{\Sph}{\mathbb S}

\newcommand{\bbbz}{\mathbb Z}

\newcommand{\R}{\mathbb R}

\newtheorem{theorem}{Theorem}
\newtheorem*{theorem*}{Theorem}

\theoremstyle{definition}
\newtheorem{remark}[theorem]{Remark}

\def\mvint_#1{\mathchoice
          {\mathop{\vrule width 6pt height 3 pt depth -2.5pt
                  \kern -9pt \intop}\limits_{\kern -3pt #1}}%
          {\mathop{\vrule width 5pt height 3 pt depth -2.6pt
                  \kern -6pt \intop}\nolimits_{#1}}%
          {\mathop{\vrule width 5pt height 3 pt depth -2.6pt
                  \kern -6pt \intop}\nolimits_{#1}}%
          {\mathop{\vrule width 5pt height 3 pt depth -2.6pt
                  \kern -6pt \intop}\nolimits_{#1}}}
\usepackage{setspace}
\title{Linking topological spheres}
\author[Haj\l{}asz]{Piotr Haj\l{}asz}
\address{Piotr Haj\l{}asz,\newline \indent Department of Mathematics, University of Pittsburgh, \newline \indent 301 Thackeray Hall, Pittsburgh,
Pennsylvania 15260}
\email{hajlasz@pitt.edu}
\thanks{P.H. was supported by NSF grant DMS-1800457.}


\begin{document}

\sloppy

\maketitle

\begin{center}
{\em Dedicated to MathOverflow}
\end{center}

\begin{abstract}
There is a topological embedding $\iota:\Sph^1\to\R^5$ such that
$\pi_3(\R^5\setminus\iota(\Sph^1))=0$. Therefore, no $3$-sphere can be linked with $\iota(\Sph^1)$.
\end{abstract}

If $\Sph^k$ is smoothly embedded in $\R^n$, $n>k+1$, then there is an embedding of $\Sph^{n-k-1}$ so that the two spheres are linked with the liking number equal $1$. 
Now, if $\iota:\Sph^k\to\R^{n}$ is an arbitrary topological embedding, then by the Alexander duality $H_{n-k-1}(\R^n\setminus \iota(\Sph^k))=\bbbz$, so the homology of the complement is the same for all embeddings.  Therefore, one might expect that $\iota(\Sph^k)$ is linked with some ($n-k-1$)-dimensional sphere embedded into the complement of $\iota(\Sph^k)$, so that this embedding of $\mathbb{S}^{n-k-1}$ gives a generator of $H_{n-k-1}(\R^n\setminus \iota(\Sph^k))=\bbbz$. This is the case when $k=n-2$. Indeed, then $n-k-1=1$ and the existence of a linked sphere follows from the fact that that $H_1$ is the abelianization of $\pi_1$, so a generator of $H_1=\bbbz$ is given by a mapping from a circle; this mapping can be approximated by an embedding leading to an embedded $\Sph^1$ that is linked with $\iota(\Sph^k)$ with the linking number $1$. However, the case when $n-k-1>1$ is more subtle.

The purpose of this short note is to construct a somewhat surprising example: there is a topological embedding $\iota_1:\Sph^1\to\R^5$ such that for every embedding $\iota_2:\Sph^3\to\R^5\setminus \iota_1(\Sph^1)$, the spheres $\iota_1(\Sph^1)$ and $\iota_2(\Sph^3)$ are unlinked. More precisely
$\pi_3(\R^5\setminus \iota_1(\Sph^1))=0$ so the embedding of $\Sph^3$ in the complement of $\iota_1(\Sph^1)$ is always homotopic to a constant map. 

The result is a simple consequence of the celebrated Cannon-Edwards theorem \cite{cannon,edwards0,edwards} according to which the double suspension of the integer homology sphere is a topological sphere. This deep and counterintuitive result is often used to construct various counterexamples that are similar in the spirit to the one presented below.

Since the linking number of highly non-smooth topological spheres has recently been used in geometric analysis (cf.\ \cite{GH,HMS,henclm,ziemer}), the author hopes that this example will be of some interest.

Using a one point compactification of $\R^5$ it suffices to prove the following result.
\begin{theorem}
\label{main}
There is a topological embedding $\iota:\Sph^1\to\Sph^5$ such that $\pi_3(\Sph^5\setminus\iota(\Sph^1))=0$.
\end{theorem}
\begin{proof}
It is well known that there are $3$-dimensional integer homology spheres whose universal cover is $\R^3$. For example, there are particular constructions in \cite{baldwins,brockd,homl} of hyperbolic integer homology spheres. Note that the universal cover of a hyperbolic $3$-manifold is the hyperbolic space that is homeomorphic to $\R^3$. 
Other examples are listed in \cite{M1}.
Let $\cM$ be such an integer homology sphere. Since the universal cover of $\cM$ is contractible, $\pi_3(\cM)=0$. According to the celebrated theorem of Cannon and Edwards \cite{cannon,edwards0,edwards}, the double suspension of an integer homology sphere is homeomorphic to a topological sphere. Let $h:S^2\cM\to\Sph^5$ be such a homeomorphism.
$\cM$ is a deformation retract of the complement of the vertices of the suspension $S\cM$. 
Therefore, $\cM$ is also a deformation retract of the complement of the suspension of the vertices in $S^2\cM$. Denote the suspension of the vertices by $X$, so $\cM$ is a deformation retract of $S^2\cM\setminus X$ and hence $\pi_3(S^2\cM\setminus X)=0$. 
$X$ is homeomorphic to $\Sph^1$. If $g:\Sph^1\to X$ is a homeomorphism, then $\iota=h\circ g:\Sph^1\to\Sph^5$ is a topological embedding and clearly $\pi_3(\Sph^5\setminus\iota(\Sph^1))=\pi_3(\Sph^5\setminus h(X))=\pi_3(S^2\cM\setminus X)=0$.
The proof is complete.
\end{proof}

\begin{remark}
Taking the higher suspensions of the hyperbolic integer homology sphere we obtain embeddings of $\iota:\Sph^k\to\Sph^{k+4}$, $k\geq 1$, such that $\pi_3(\Sph^{k+4}\setminus \iota(\Sph^k))=0$, so 
no $3$-sphere can be linked with $\iota(\Sph^k)$.
\end{remark}

\begin{remark}
If in the above construction we replace the hyperbolic homology sphere by the standard Poincar\'e homology sphere $\cM$, the linking number of the embedding of $\mathbb{S}^3$ to $\Sph^5\setminus\iota(\Sph^1)$ will be the multiplicity of $120$. Indeed, 
$|\pi_1(\cM)|=120$ and hence the degree of a map $f:\mathbb{S}^3\to M^3$ is a multiple of $|\pi_1(\cM)|=120$.
\end{remark}

\noindent{\bf Acknowledgements.}
The author would like to express his deepest gratitude to the Mathoverflow community; without their help this short note would have never been written. In particular he would like to thank Ian Agol, Jason DeBlois, Neil Hoffman, John Klein, Thilo Kuessner, Andrew Putman, Daniel Ruberman, see \cite{M1,M2}. The author is also grateful to the referee for the valuable comments.

\end{document}